\documentclass[a4paper,10pt]{scrartcl}
\usepackage[utf8]{inputenc}
\usepackage[english]{babel}
\usepackage[fixlanguage]{babelbib}
\usepackage{cite}
\usepackage{amssymb, amsmath, amstext, amsopn, amsthm, amscd, amsxtra, amsfonts}
\usepackage{yfonts, mathrsfs}
\usepackage{paralist}

\usepackage{enumerate}

\usepackage[all]{xy}

\usepackage[T1]{fontenc}
\usepackage{ifthen}
\usepackage{colonequals}

\usepackage[bookmarks,bookmarksnumbered,plainpages=false]{hyperref}

\swapnumbers
\newtheoremstyle{standard}
 {16pt}  
 {16pt}  
 {}  
 {}  
 {\bfseries}
 {}  
 { } 
 {{\thmname{#1~}}{\thmnumber{#2.}}\thmnote{~(#3)}} 

\newtheoremstyle{kursiv}
 {16pt}  
 {16pt}  
 {\itshape}  
 {}  
 {\bfseries}
 {}  
 { } 
 {{\thmname{#1~}}{\thmnumber{#2.}}\thmnote{~(#3)}} 

\theoremstyle{standard}
\newtheorem{defn} [subsection]{Definition}

\newtheorem{rem}   [subsection]{Remark}

\newtheorem{setup} [subsection]{}

\theoremstyle{definition}

\theoremstyle{kursiv}

\newtheorem*{thm*}{Theorem}
\newtheorem*{thmA}{Theorem A}
\newtheorem*{thmB}{Theorem B}
\newtheorem*{thmC}{Theorem C}
\newtheorem*{thmD}{Theorem D}
\newtheorem{prop} [subsection]{Proposition}

\newtheorem{lem} [subsection]{Lemma}

\setdefaultenum{\normalfont (a)}{i.}{A.}{1.}

\newcommand{\im}{\mathrm{im}}

\newcommand{\N}{\mathbb{N}}

\newcommand{\R}{\mathbb{R}}
\newcommand{\K}{\mathbb{K}}

\renewcommand{\epsilon}{\varepsilon}

\newcommand{\setm}[2]{\left\{\, #1 \middle\vert #2\,\right\}}

\newcommand{\abs}[1]{\left| #1 \right|}

\newcommand{\caseelse}[3]{\left\{ \begin{array}{r l} #1 & \  \mathrm{if} \  #2 \\ #3 & \ \mathrm{else.} \end{array} \right.}

\DeclareSymbolFont{bbold}{U}{bbold}{m}{n}
\DeclareSymbolFontAlphabet{\mathbbold}{bbold}

\newcommand{\cA}{\ensuremath{\mathcal{A}}}

\newcommand{\func}[5]{#1 \colon #2 \rightarrow #3 , #4 \mapsto #5}
\newcommand{\smfunc}[3]{#1 \colon #2 \rightarrow #3}
\newcommand{\nnfunc}[4]{#1 \rightarrow #2 : #3 \mapsto #4}
\newcommand{\smset}[1]{ \left\{ #1 \right\} }

\newcommand{\seqa}[1]{\left( #1 \right)_{\alpha\in A}}

\newcommand{\Frechet}{Fr\'echet }

\newcommand{\Func}[5]{
\begin{array}{rl}
   #1 : #2 & \longrightarrow #3   \\
        #4 & \longmapsto     #5   \\
\end{array}
}

\DeclareMathOperator{\supp}{supp}
\DeclareMathOperator{\embdim}{embdim}
\DeclareMathOperator{\conncomp}{conncomp}

\title{Smooth embeddings of the Long Line and other non-paracompact manifolds into locally convex spaces} 
 \author{  Rafael Dahmen\footnote{Technische Universit\"at Darmstadt, Germany. 
 \href{mailto:dahmen@mathematik.tu-darmstadt.de}{dahmen@mathematik.tu-darmstadt.de}}%
 \date{March 26, 2015}
 }
 
\begin{document}

\maketitle

\begin{abstract}
 We show that every real finite dimensional Hausdorff (not necessarily paracompact, not necessarily second countable) $C^r$-manifold can be embedded into a weakly complete vector space, i.e.~a locally convex topological vector space of the form ${\mathbb R}^I$ for an uncountable index set $I$ and determine the minimal cardinality of $I$ for which such an embedding is possible.
\end{abstract}

\medskip


\medskip

 (Secondary)


\section{Introduction and statement of the results}

We review the classical Theorem of Whitney (see e.g. \cite[Theorem 1]{MR1503303}, \cite[Theorem 6.3]{MR0470974}, or \cite[Theorem 2.2]{MR1225100}):

\begin{thm*}[Whitney]
 Let $M$ be a $d$-dimensional second countable Hausdorff $C^r$-manifold ($d\geq1$). Then there exists a $C^r$-embedding into $\R^{2d}$.
\end{thm*}
The conditions (Second Countability and the Hausdorff-property) are obviously also necessary, since every Euclidean space $\R^n$ is second countable and Hausdorff and so are all of its subsets. The dimension $2d$ in the Theorem is sharp in the sense that whenever $d$ is a power of two, i.e.~$d=2^k$, there is a $2^k$-dimensional second countable Hausdorff manifold which can not be embedded into $\R^{2^k+1}$.

Unfortunately, not every Hausdorff $C^r$-manifold is second countable. Perhaps the easiest connected manifold where the second axiom of countability does not hold, is the \emph{Long Line} and its relative the \emph{Open Long Ray} (see e.g.~\cite{MR0101321}). For the reader's convenience, we will recall the definition:
\newcommand{\ClosedLongRay}{\mathbb{L}^+_C}
\newcommand{\OpenLongRay}{\mathbb{L}^+}
\newcommand{\LongLine}{\mathbb{L} }

\begin{defn}[The Alexandroff Long Line]					\label{defn: LongLine}
 \begin{itemize}
  \item [(a)]  Let $\omega_1$ be the first uncountable ordinal. The product 
		$\ClosedLongRay  :=\omega_1\times \left[0,1\right[$, endowed with the lexicographical (total) order, becomes a topological space with the order topology, called the \emph{Closed Long Ray}. This space is a connected (Hausdorff) one-dimensional topological manifold with boundary $\smset{(0,0)}$.
  \item [(b)] To obtain a manifold without boundary, one removes this boundary point.
		The resulting open set $\OpenLongRay:=\ClosedLongRay\setminus\smset{(0,0)}$ is called the \emph{Open Long Ray}.
  \item [(c)] A different way to obtain a manifold without boundary is to consider two copies of the Closed Long Ray and glue them together at their boundary points. The resulting one-dimensional topological manifold $\LongLine$ is called the \emph{Long Line}\footnote{Some authors use the term Long Line for what we call here Long Ray.}.
 \end{itemize}
\end{defn}
 The spaces $\LongLine,\OpenLongRay,$ and $\ClosedLongRay$ are locally metrizable but since countable subsets are always bounded, none of the three spaces is separable. So, in particular, they are not second countable. Since we are only considering manifolds without boundary in this paper, for us only the Open Long Ray and the Long Line are interesting. 

 It is known that there exist $C^r$-structures on $\OpenLongRay$ and on $\LongLine$ for each $r\in\N\cup\smset{\infty}$. They are however not unique up to diffeomorphism. For example, there are $2^{\aleph_1}$ pairwise non-diffeomorphic $C^\infty$-structures on $\LongLine$ (cf. \cite{MR1164707}). 

 The Long Line and the Open Long Ray are by far not the only interesting examples of non-second countable manifolds. A famous two dimensional example (which has been known even before the Long Line) is the so called  \emph{Pr\"ufer manifold} (see \cite{51.0273.01} for a definition).

Since these manifolds fail to be second countable they cannot be embedded into a finite dimensional vector space. However, one can ask the question if it is possible to embed them into an infinite dimensional space. Of course, before answering this question, one has to say concretely what this should mean as there are different, non-equivalent notions of  differential calculus in infinite dimensional spaces: We use the setting of Michal-Bastiani, based on Keller's $C_c^r$-calculus (see \cite{MR1911979}, \cite{keller1974}, \cite{MR830252} and \cite{MR2261066}). This setting allows us to work with $C^r$-maps between arbitrary locally convex spaces, as long as they are Hausdorff. A manifold modeled on a locally convex space can be defined via charts the usual way, there is a natural concept of a $C^r$-submanifold generalizing the concept in finite dimensional Euclidean space. Important examples of locally convex spaces are Hilbert spaces, Banach spaces, \Frechet spaces and infinite products of such spaces such as $\R^I$ for an arbitrary index set $I$. Since this differentiable calculus explicitly requires the Hausdorff property, we will not consider embeddings of non-Hausdorff manifolds, although there are interesting examples of those (occurring naturally as quotients of Hausdorff manifolds, e.g. leaf spaces of foliations etc.) Our first result is the following:

\begin{thmA}\ \\
 Let $r\in\smset{1,2,\ldots}\cup\smset{\infty}$.
 Let $M$ be a finite dimensional Hausdorff $C^r$-manifold (not necessary second countable). Then there exists a set $I$ such that $M$ can be $C^r$-embedded into the locally convex topological vector space $\R^I$.
\end{thmA}
%

A locally convex vector space of the type $\R^I$ is called \emph{weakly complete vector space} (see \cite[Appendix C]{HopfPaper} or \cite[Appendix 2]{MR2337107}). These weakly complete spaces form a good generalization of finite dimensional vector spaces. The cardinality of the set $I$, sometimes called the \emph{weakly complete dimension}, is a topological invariant of $\R^I$ (see Lemma \ref{lem: weakly complete dimension}). This gives rise to the following question: Given a finite dimensional $C^r$-manifold $M$, what is the minimal weakly complete dimension which is necessary to embed $M$?
Unfortunately, we will not be able to answer this question completely, but we will give upper and lower bounds and if the Continuum Hypothesis is true, then we have a complete answer:

\begin{thmB}\ \\
 Let $M$ be finite dimensional Hausdorff $C^r$-manifold (not necessary second countable). The \emph{embedding dimension} of $M$ is defined as
 \[
  \embdim(M):=\min\setm{\abs{I}\ }{\hbox{There is a $C^r$-embedding of $M$ into $\R^I$}}.
 \]
 Furthermore, let $\conncomp(M)$ denote the number of connected components of $M$. Then the following holds:
 \begin{itemize}
  \item [(a)] The cardinal $\embdim(M)$ is finite if and only if $M$ is second countable.
  \item [(b)] If $M$ is not second countable and $\conncomp(M)\leq 2^{\aleph_0}$ (the  continuum), then 
            \[
             \aleph_1 \leq \embdim(M) \leq 2^{\aleph_0}.
            \]
  \item [(c)] If $\conncomp(M)\geq 2^{\aleph_0}$, then
            \[
             \embdim(M) = \conncomp(M).
            \]
 \end{itemize}
 In particular, $\embdim(M)$ is never equal to $\aleph_0$.
\end{thmB}
Of course, if the continuum hypothesis $\aleph_1 = 2^{\aleph_0}$ holds, then this theorem gives an exact answer.

Now for something completely different: By a theorem of Kneser (see \cite{MR0101321}), there is a real-analytic ($C^\omega_\R$-) structure on the Long Line. (in fact, there are infinitely many non-equivalent of them (see e.g.~\cite[Satz 1]{MR0113228})). For second countable manifolds, there is an analogue of Whitney's Theorem for real analytic manifolds, stating that every second countable Hausdorff $C^\omega_\R$-manifold can be embedded analytically into a Euclidean space (see e.g.\cite[8.2.3]{MR2975791}).
This rises the question whether we can embed a real analytic Long Line into a space of the form $\R^I$. We answer this question in the negative:

\begin{thmC} \ \\
 Let $M$ be the Long Line $\LongLine$ or the Open Long Ray $\OpenLongRay$ with a $C^\omega_\R$-structure. Then it is not possible to embed $M$ as a $C^\omega_\R$-submanifold into any locally convex topological vector space. In fact, every $C^\omega_\R$-map from $M$ into any locally convex topological vector space is constant.
\end{thmC}

Lastly, we will address the question whether the embeddings we constructed in Theorem A have closed image. It is a well-known (and easy to show) fact that each $C^r$-submanifold of $\R^n$ is $C^r$-diffeomorphic to a \emph{closed} $C^r$-submanifold of $\R^{n+1}$ . Hence, every submanifold of a finite dimensional space can be regarded as a \emph{closed} submanifold of a (possibly bigger) finite dimensional space\footnote{It is also possible to construct the Whitney embedding already in such a way that the image is closed in $\R^{2d}$.}. The question whether this also holds in our setting, i.e.~whether the Long Line is diffeomorphic to a closed submanifold of $\R^I$ is answered to the negative:

\begin{thmD} \ \\ 
 Let $M$ be the Long Line $\LongLine$ or the Open Long Ray $\OpenLongRay$ and let $E$ be any \emph{complete} locally convex vector space. Then $M$ is \emph{not} homeomorphic to a \emph{closed} subset of $E$. In particular, $M$ cannot be a closed submanifold of $E$.
\end{thmD} 

If one allows non-complete locally convex spaces, then there is a way to embed $M$ \emph{topologically} as a closed subset into a non-complete locally convex space (see Remark \ref{rem_noncomplete_emb}). However, this embedding is merely continuous but fails to be $C^r$. It is not known to the author if there is a way to construct a $C^r$-embedding with a closed image in a locally convex space.

\section{Construction of the embedding}
 \begin{setup}
  Whenever we speak of $C^r$-mappings and $C^r$-manifolds, we refer to the locally convex differential calculus by Michal-Bastiani. Details can be found in \cite{MR1911979}, \cite{keller1974}, \cite{MR830252} and \cite{MR2261066}.
   For the special case that a function $f$ is defined on $\Omega\subseteq\R^d$, this is equivalent to the notion that $f$ is $r$-times partially differentiable and that all $\partial^\alpha f$ are continuous (see for example \cite[Appendix A.3]{MR2952176})
  All manifolds are assumed to be Hausdorff but we do not assume that they are connected (and of course we will not assume that they are paracompact or even second countable).
 \end{setup}

%
 \begin{setup}[$C^r$-submanifolds]
  Let $r\in\N\cup\smset{\infty}$.
  Let $M$ be a $C^r$-manifold modeled on a locally convex space $E$ and $F\subseteq E$ be a closed vector subspace. A subset $N\subseteq M$ is called a \emph{$C^r$-submanifold of $M$ modeled on $F$} if for each point $p\in N$ there is a $C^r$-diffeomorphism $\smfunc{\phi}{U_\phi}{V_\phi}$ with $U_\phi\subseteq M$ and $V_\phi\subseteq E$ open such that $\phi(U_\phi\cap N)=V_\phi\cap F$.

  The $C^r$-submanifold $N$ then carries a natural structure of a $C^r$-manifold modeled on the vector space $F$. It should be noted that although $F$ is assumed to be closed in $E$, we do not assume that $N$ is a closed subset of $M$.
 \end{setup}

 \begin{setup}[$C^r$-embeddings]
  Let $M$ and $N$ be locally convex manifolds and let $\smfunc{f}{M}{N}$ be a $C^r$-map. We call $f$ a \emph{$C^r$-embedding} if $f(M)$ is a submanifold of $N$ and $\nnfunc{M}{f(M)}{x}{f(x)}$ is a $C^r$-diffeomorphism. Our goal is to show that for each finite dimensional $M$ there is a set $I$ such that there is a $C^r$-embedding $\smfunc{f}{M}{\R^I}$.
 \end{setup}

 \begin{setup}[Immersions]
  One reason why locally convex differential calculus is more involved than in finite dimensions is that there are at least two different notions of \emph{immersions}: 
  Let $\smfunc{f}{M}{N}$ be a $C^r$-map between locally convex manifolds. For the sake of this article, let us call $f$ a \emph{weak immersion} if the tangent map $\smfunc{T_af}{T_aM}{T_{f(a)}N}$ at each point $a\in M$ is injective. We call $f$ a \emph{strong immersion} if every $a\in M$ has an open neighborhood $U\subseteq M$ such that $\smfunc{f|_U}{U}{N}$ is a $C^r$-embedding.
 \end{setup}
 
 \begin{lem}[Immersion Lemma]										\label{lem: immersion_lemma}
  For a $C^r$-map $\smfunc{f}{M}{N}$ on a \emph{finite dimensional} manifold $M$ and a locally convex manifold $N$, the following are equivalent:
  \begin{itemize}
   \item [(a)] $f$ is a weak immersion.
   \item [(b)] $f$ is a strong immersion.
  \end{itemize}
 \end{lem}
 \begin{proof}
  Since both properties are local, we may assume that $\smfunc{f}{\Omega}{F}$ is a $C^r$-map, where $M=\Omega$ is an open subset of $\R^d$, while $N=F$ is a (Hausdorff) locally convex topological vector space.

  Since the implication (b) $\Rightarrow$ (a) holds trivially, even without the domain being finite dimensional, we will show (a)$\Rightarrow$(b). To this end, let $a\in \Omega$ be fixed. We assume that $\smfunc{T_af}{\R^d}{F}$ is injective. This means that $V:=\im T_af$ is a $d$-dimensional vector subspace of $F$. Since finite dimensional vector subspaces in locally convex spaces are always complemented (see e.g.~\cite[Corollary 2 in Chapter 7.2]{MR632257}), we may assume that $F=V\oplus W$ with a closed vector subspace $W\subseteq F$. We obtain the projections $\smfunc{\pi_V}{F}{V}$ and $\smfunc{\pi_W}{F}{W}$. It is easy to check that the tangent map of $\smfunc{\pi_V\circ f}{\Omega}{V}$ is an invertible linear map between the $d$-dimensional real vector spaces $\R^d$ and $V$. Hence, by the usual Inverse Function Theorem for $C^r$-maps, there exists a small neighborhood $\Omega'\subseteq \Omega$ of $a$ such that $\pi_V\circ f$ maps $\Omega'$ diffeomorphic onto $\Omega_V:=(\pi_V\circ f)(\Omega')$. The inverse map will be denoted by $\smfunc{g}{\Omega_V}{\Omega'}$.

  Now, the function $\smfunc{\phi:=\pi_W\circ f \circ g}{\Omega_V}{W}$ is a $C^r$-map and hence, its graph
  \[
   G:=\setm{(v,w)\in V\oplus W}{v\in\Omega_V \hbox{ and } w=\phi(v)}
  \]
  is a submanifold of $V\oplus W = E$. It is now easy to check that the image of $f|_{\Omega'}$ is the set $G$ and that the map is a $C^r$-diffeomorphism onto its image.
 \end{proof}

 \begin{rem}
  While this Lemma holds for finite dimensional $M$, it has to be said that statements like these fail to hold if the domain is infinite dimensional, in particular beyond Banach space theory due to the lack of an Inverse Function Theorem.

  This main essence of this last lemma (with slightly different definitions and vocabulary) can also be found in \cite{HelgeSubmersionsAndImmersions} which provides a good overview of immersions and submersions in infinite dimensional locally convex spaces.
 \end{rem}

 \begin{prop}												\label{prop: topological immersion}
  Let $\smfunc{f}{M}{N}$ be a map between locally convex $C^r$-manifolds $M$ and $N$. Then $f$ is a $C^r$-embedding if and only if $f$ is a topological embedding and a strong $C^r$-immersion.
 \end{prop}
 \begin{proof}
  Having the right (strong) definition of \emph{immersion}, this proposition is easy to show:
  Let $a\in M$. Then there exists an open neighborhood $U\subseteq M$ of $a$ such that $\smfunc{f|_U}{U}{N}$ is a $C^r$-embedding. Since $f$ is a topological embedding, the image of $U$ under $f$ is open in $f(M)$,i.e.~there exists an open neighborhood $V$ of $f(a)$ in $V$ such that $f(U)=V\cap f(M)$.  
 \end{proof}
 
 Now, we are ready to show that every finite dimensional Hausdorff $C^r$-manifold admits a $C^r$-embedding into a locally convex space of the type $\R^I$ for an index set $I$:
 \begin{proof}[Proof of Theorem A]
  Let $M$ be a finite dimensional $C^r$-manifold and let
  \[
   I:=C_c^r(M,\R)
  \]
  denote the set of all compactly supported $C^r$-functions on $M$. We define the following map
  \[
   \func{\Phi}{M}{\R^I}{x}{(f(x))_{f\in I}.}
  \]
  This map is $C^r$ since every component is $C^r$, in particular, it is continuous.

  Next, we show that it is a topological embedding. Let $a\in M$ be a point and let $\seqa{a_\alpha}$ be a net in $M$ with the property that $\seqa{\Phi(a_\alpha)}$ converges to $\Phi(a)$. We will show that $\seqa{a_\alpha}$ converges to $a$. To this end, let $U\subseteq M$ be an open neighborhood of $a$. It is possible to construct a function $f\in I$ such that $\supp(f)\subseteq U$ and $f(a)=1$. Since $\seqa{\Phi(a_\alpha)}$ converges to $\Phi(a)$ in the product space $\R^I$ and since projection onto the $f$-th component is continuous, we obtain that $\seqa{f(a_\alpha)}$ converges in $\R$ to $f(a)=1$. Hence, there is an $\alpha_0$ such that $f(a_\alpha)>0$ for all $\alpha\geq\alpha_0$. This implies that $a_\alpha\in U$ for all $\alpha\geq\alpha_0$ and hence, $\Phi$ is a topological embedding.

  It remains to show that $\Phi$ satisfies part (a) of Lemma \ref{lem: immersion_lemma}. Then, together with Proposition \ref{prop: topological immersion}, the assertion follows.

  To this end, let $a\in M$ be fixed. It remains to show that the linear map $\smfunc{T_a\Phi}{T_aM}{\R^I}$ is injective. Since this is a local property, we may assume that $M$ is an open $0$-neighborhood in $\R^d$ and that $a=0$.

  We obtain the following formula for the linear map:
  \[
   \func{T_a\Phi}{\R^d}{\R^I}{v}{(T_af(v))_{f\in I}.}
  \]
  Let $v\in \ker T_a\Phi$ and fix a linear map $\smfunc{\lambda}{\R^d}{\R}$. We define the following function $f_0\in I$ via
  \[
   \func{f_0}{M}{\R}{x}{\theta(x)\cdot \lambda(x),}
  \]
  where $\smfunc{\theta}{\R^d}{\R}$ is a suitable $C^r$-function with compact support in $M$ and the property that $\theta(x)=1$ for all $x$ in a small neighborhood of $0$.
  Since $v\in \ker T_a\Phi(v)$, this implies that $v\in \ker T_af$ for all $f\in I$. In particular, we have that $T_af_0(v)=0$. But since $f_0$ is equal to $\lambda$ in a neighborhood of $0$, this implies that $\lambda(v)=0$. Since $\lambda$ was arbitrary, this implies that $v=0$. This finishes the proof.
 \end{proof}

 \section{The embedding dimension}
 In this section we will give a proof of Theorem B stated in the introduction.

 \begin{lem}[Weight = Weakly Complete Dimension]								\label{lem: weakly complete dimension}
  Let $\R^I$ be a weakly complete vector space with $\abs{I}$ infinite. Then the cardinality of $I$ is a topological invariant of the space, i.e.~it can be computed using only the topology of $\R^I$ and not the vector space structure:
  \begin{itemize}
   \item [(a)] The cardinal $\abs{I}$ is the minimal cardinality of a basis of the topology of $\R^I$, i.e.~$\abs{I}$ is the \emph{weight} of the topological space $\R^I$.
   \item [(b)] The cardinal $\abs{I}$ is the maximal cardinality of a discrete subset of the space $\R^I$.
  \end{itemize}
 \end{lem}
 \begin{proof}
  If a topological space has a topological basis of cardinality at most $\alpha$. Then each subset has the same property. In particular, each discrete subset has a topological basis of cardinality at most $\alpha$ which implies that the discrete subset itself has at most $\alpha$ many elements. This shows that the maximal cardinality of a discrete subset is less than or equal to the minimal cardinality of a basis.

  The product topology on $\R^I$ has a basis of the topology of $\abs{I}$ many sets (using that $I$ is infinite). This shows that the minimal cardinality of a basis is bounded above by $\abs{I}$.

  Lastly, the set $\setm{e_i}{i\in I}$ of unit vectors in $\R^I$ is discrete, showing that $I$ is less than or equal to the maximal cardinality of a discrete subset. Putting these arguments together, the claim follows.
 \end{proof}

 \begin{rem}
  In the case that $\abs{I}=d$ is finite, the weakly complete dimension $d$ of $\R^d$ is no longer equal to the weight of the spaces $\R^d$. However, $d$ is still uniquely determined by the topology of $\R^d$ by the invariance of dimension from algebraic topology (see e.g.~\cite[Theorem 2.26]{MR1867354}). However, we will not need this fact here.
 \end{rem}

 We will start with a lemma which can be found in \cite[Theorem (i)]{MR710241}
 \begin{lem}								\label{lem: continuum atlas}
  Let $M$ be a finite dimensional topological \emph{connected} Hausdorff manifold. Then $M$ admits an atlas of continuum cardinality.
 \end{lem}
 Using this lemma, we can easily proof the following:
 \begin{lem}								\label{lem: continuum seperable cover}
  Let $M$ be a finite dimensional topological connected Hausdorff manifold. Then $M$ admits an open cover $\left( U_j \right)_{j\in J}$ which is stable under finite unions and such that each $U_j$ is separable and such that  $\abs{J}\leq 2^{\aleph_0}$.
 \end{lem}
 \begin{proof}
  By Lemma \ref{lem: continuum atlas} we know that $M$ has an atlas $\cA$ with $\abs{\cA}\leq 2^{\aleph_0}$. Every chart domain of $\cA$ is homeomorphic to a subset of $\R^d$ and hence separable. Unfortunately, the union of two chart domains is in general not a chart domain. Hence, we consider all finite unions of chart domains. The finite union of open separable sets is open and separable. If the cardinality of the atlas is infinite, it will not increase by allowing finite unions of elements. Hence, it will still be bounded above by the continuum.
 \end{proof}
 We will now give the proof of Theorem B:
 \begin{proof}[Proof of Theorem B]
  Part (a) is just the classical Theorem of Whitney for $C^r$-maps stated in the introduction.

  For the proof of part (b), assume that $M$ is not second countable. Since $\K^{\aleph_0}\cong\K^\N$ is a separable \Frechet space it is second countable. So, $M$ cannot be homeomorphic to a subset of $\K^{\aleph_0}$. Hence, $\embdim(M)\geq \aleph_1$.

  Assume now that $\conncomp(M)=1$, i.e.~$M$ is connected. Let $\left(U_j\right)_{j\in J}$ be the open cover from Lemma \ref{lem: continuum seperable cover}.
  For each $j\in J$, let $E_j$ be the space of all $f\in C^r_c(M,\R)$ such that $\supp(f)\subseteq U_j$. Every support of a function $f\in C^r_c(M,\R)$ is compact and hence can be covered by finitely many $U_j$. Since the system $(U_j)_j$ is directed, we may conclude that for each $f\in C^r_c(M,\R)$ there is a $j\in J$ such that $\supp(f)\subseteq U_j$, i.e.~
  \[
   C^r_c(M,\R) = \bigcup_{j\in J} E_j.
  \]
  Now, every $U_j$ is separable, i.e.~there is a dense countable set $D_j\subseteq U_j$. Each function $f\in E_j$ is in particular continuous and hence uniquely determined by its values on the dense subset $U_j$, yielding an injective map
  \[
   \nnfunc{E_j}{\R^{D_j}}{f}{f|_{D_j}}
  \]
  In the proof of Theorem A, we saw that $M$ can be embedded in $\R^I$ with $I:=C^r_c(M,\R)$. This allows us to estimate the embedding dimension of $M$  as follows:
  \begin{align*}
   \embdim(M)\leq \abs{I}	  & = 		\abs{C^r_c(M,\R)}
				\\& = 		\abs{\bigcup_{j\in J} E_j}
				\\& \leq 	\sum_{j\in J} \abs{E_j} 
				\\& \leq 	\sum_{j\in J} \abs{\R^{D_j}}
				\\& =    	\sum_{j\in J} \abs{\R}^{\abs{D_j}}
				\\& \leq    	\sum_{j\in J} \left( 2^{\aleph_0}\right)^{\aleph_0}
				\\& =    	\sum_{j\in J} 2^{\aleph_0}
				\\& =    	\abs{J} \cdot 2^{\aleph_0}
				\\& \leq    	2^{\aleph_0} \cdot 2^{\aleph_0}
				\\& =    	2^{\aleph_0}.
  \end{align*}
  This finished the proof for the case that $\conncomp(M)=1$.

  Now, for case $\conncomp(M)\geq1$:
  Let $\left( M_\alpha\right)_{\alpha\in A}$ be the family of connected components of $M$. By the preceding calculation, we know that each $M_\alpha$ admits a $C^r$-embedding
  \[
   \smfunc{\Phi_\alpha}{M_\alpha}{\R^{2^{\aleph_0}}.}
  \]
  For each $\alpha\in A$, we define the function
  \[
   \func{e_\alpha}{A}{\R}{\beta}{\delta_{\alpha,\beta}:=\caseelse{1}{\beta=\alpha}{0}.}
  \]
  As $e_\alpha$ is a function from the index set $A$ to $\R$, we have that $e_\alpha$ belongs to the weakly complete vector space $\R^A$. It is easy to see that $\setm{e_\alpha}{\alpha\in A}$ is a discrete subset of $\R^A$.

  Now, we are able to define the embedding of $M$:
  \[
   \Func{\Phi}{M}{\R^A \times \R^{2^{\aleph_0}}}{x \in M_\alpha}{\left(e_\alpha, \Phi_{\alpha}(x)\right).}
  \]
  Is it straightforward to check that this is a $C^r$-embedding. Using this embedding, we obtain an upper bound for the embedding dimension:
  \[
   \embdim(M) \leq \abs{A}\cdot 2^{\aleph_0} = \conncomp(m) \cdot 2^{\aleph_0} = \max(\conncomp(A) , 2^{\aleph_0}),
  \]
  where the last equality used the well-known fact in cardinal arithmetic that the product of two cardinals is equal to the maximum if both are nonzero and at least one of them is infinite.

  It remains to show that $\embdim(M)\geq \conncomp(M)$. To this end, we chose from each connected component $M_\alpha$ one element $x_\alpha \in M_\alpha$. Then it is easy to see that the set $\setm{x_\alpha}{\alpha\in A}$ is discrete in $M$. This implies that $\R^{\embdim(M)}$ has a discrete subset of cardinality $\abs{A}=\conncomp(M)$. By Lemma \ref{lem: weakly complete dimension}, the cardinality of a discrete subset of a weakly complete space is always bounded above by the weakly complete dimension of the surrounding space. Hence $\conncomp(M)\leq \embdim(M)$. This finishes the proof of Theorem B.
 \end{proof}

 \section{Analytic embeddings}
  In this section we will give a proof of Theorem C stated in the introduction. To this end, let $M$ be either the Long Line or the Open Long Ray, together with one of the $C^\omega_\R$-structures on it. Let $E$ be any locally convex space and consider a $C^\omega_\R$-map $\smfunc{f}{M}{E}$. We will show that $f$ is constant. To this end, let $\smfunc{\lambda}{E}{\R}$ be any continuous linear functional on $E$. Continuous linear maps are always analytic, so are compositions of real analytic maps (see \cite[Proposition 2.8]{MR1911979}). This implies that $\smfunc{\lambda\circ f}{M}{\R}$ is a $C^\omega_\R$-function on $M$.
  However, a well-known fact about the Long Line (and the Open Long Ray) is that every continuous function becomes eventually constant (see e.g..\cite[Theorem 7.7]{MR1164707}). So, from a point onwards, $\lambda\circ f$ will be constant and by the Identity Theorem for analytic functions (and the fact that $M$ is connected), this implies that $\lambda\circ f$ is globally constant on $M$.

Since the functional $\lambda$ was arbitrary and by Hahn-Banach, the continuous linear functionals separate the points of $E$, it follows that $\smfunc{f}{M}{E}$ is constant. So in particular, $f$ cannot be an embedding.

 \newcommand{\PP}{(\boxtimes)}

 \section{Closed embeddings}
  In this section we will give a proof of Theorem D stated in the introduction. Only for internal use in this article, we will use the following terminology:
  \begin{defn}
   A topological space is called a $\PP$-space if every closed sequentially compact subset is compact.
  \end{defn}
  Clearly, a closed subset of a $\PP$-space is again $\PP$.
  \begin{lem}
   The Long Line and the Open Long Ray are not $\PP$-spaces.
  \end{lem}
  \begin{proof}
   Take an element $p$ in the Open Long Ray and consider the set $A$ of all elements $\geq p$.
   Then $A$ is closed and sequentially compact but not compact.
   Hence, the Long Ray is not $\PP$.

   The Long Line is closed in itself and sequentially compact but not compact. 
   Thus, it is not $\PP$.
  \end{proof}

  We will now show that a complete locally convex space always has property $\PP$. Then Theorem D follows immediately.

  \begin{lem}
    A locally convex topological vector space $E$ is $\PP$ if at least one of the following conditions is satisfied:
    \begin{itemize}
     \item $E$ is metrizable
     \item $E$ is complete
     \item $E$ is Montel
    \end{itemize}
  \end{lem}
  \begin{proof}
   If $E$ is metrizable, then every subset is metrizable. Hence, sequentially compactness is equivalent to compactness.

   Let $E$ be Montel and let $A\subseteq E$ be a closed sequentially compact subset. Then $p(A)$ is compact for every continuous seminorm $p$. Hence, $A$ is bounded. But closed and bounded subsets of Montel spaces are compact.

   Let $E$ be a complete locally convex space and let $A\subseteq E$ be closed and sequentially compact.
   Every locally convex space is isomorphic to a vector subspace of a product of Banach spaces. Thus, we may assume that
   \[
    E\subseteq \prod_{\alpha} E_\alpha,
   \]
   where each $E_\alpha$ is a Banach space. The projection $\smfunc{\pi_\alpha}{E}{E_\alpha}$ is continuous, hence $K_\alpha:=\pi_\alpha(A)$ is sequentially compact in $E_\alpha$. Since $E_\alpha$ is metrizable, each $K_\alpha$ is compact. The set $A$ is now contained in the product $\prod_{\alpha} K_\alpha$ which is compact by Tychonoff. Now, $A$ is closed in $E$ and (since $E$ is complete) $E$ is closed in the product, hence $A$ is closed in $\prod_{\alpha}E_\alpha$ and contained in the compact set $\prod_{\alpha} K_\alpha$. Thus, $A$ is compact.
  \end{proof}

  \begin{rem}
   Recall that a finite dimensional manifold is called \emph{$\omega$-bounded}, if every countable subset is relatively compact. Since every $\omega$-bounded finite dimensional manifold is sequentially compact, the exact same argument as above shows that every non-compact $\omega$-bounded manifold fails to be $\PP$ and hence cannot be embedded as a closed subset of a complete locally convex vector space.
  \end{rem}

  \begin{rem}							\label{rem_noncomplete_emb}
   If one allows locally convex spaces which are not $\PP$, then we can embed every finite dimensional $C^r$-manifold \emph{topologically} as a closed subset.

   Fix a $C^r$-manifold $M$ and consider the set $I:=C^r_c(M)\cup\smset{1_M}$ where $1_M$ denotes the constant $1$-function on $M$. Then the construction in the proof of Theorem A yields a $C^r$-embedding:
   \[
    \Func{\Phi}{M}{\R^I}{x}{(f(x))_{f\in I}.}
   \]
   Now, let $E:=\mathrm{span}(\Phi(M))$ be the real vector subspace of $\R^I$ generated by the image of $\Phi$. One can verify that $\Phi(M)$ is closed in $E$ and since $E$ carries the subspace topology of $\R^I$, the map is still a \emph{topological} embedding. This shows that $M$ can always be embedded as a closed subset in a locally convex vector space.

   Unfortunately, this map $\nnfunc{M}{E}{x}{\Phi(x)}$ will no longer be $C^r$ as a map with values in the non-closed subspace $E\subseteq\R^I$ (although it is $C^r$ as a map with values in the surrounding space $\R^I$).
   Hence, this construction does not give us a $C^r$-embedding of $M$ into $E$.

   It is not known to the author if there is a different construction such that $M$ embeds as a closed $C^r$-submanifold in a locally convex space.
  \end{rem}

\phantomsection
\addcontentsline{toc}{section}{References}
\bibliographystyle{new}
\bibliography{Literatur}

\def\polhk#1{\setbox0=\hbox{#1}{\ooalign{\hidewidth
  \lower1.5ex\hbox{`}\hidewidth\crcr\unhbox0}}}
\begin{thebibliography}{BDS15}
\providecommand{\url}[1]{\texttt{#1}}
\providecommand{\urlprefix}{URL }
\expandafter\ifx\csname urlstyle\endcsname\relax
  \providecommand{\doi}[1]{doi:\discretionary{}{}{}#1}\else
  \providecommand{\doi}{doi:\discretionary{}{}{}\begingroup
  \urlstyle{rm}\Url}\fi
\providecommand{\eprint}[2][]{\url{#2}}

\bibitem[Ada93]{MR1225100}
Adachi, M.
\newblock \emph{Embeddings and immersions}, \emph{Translations of Mathematical
  Monographs}, vol. 124 (American Mathematical Society, Providence, RI, 1993).
\newblock Translated from the 1984 Japanese original by Kiki Hudson

\bibitem[AM77]{MR0470974}
Auslander, L. and MacKenzie, R.~E.
\newblock \emph{Introduction to differentiable manifolds} (Dover Publications,
  Inc., New York, 1977).
\newblock Corrected reprinting

\bibitem[BDS15]{HopfPaper}
Bogfjellmo, G., Dahmen, R. and Schmeding, A.
\newblock \emph{Character groups of {H}opf algebras as infinite-dimensional
  {L}ie groups} 2015.
\newblock \urlprefix\url{http://arxiv.org/abs/1501.05221v2}.
\newblock \eprint{1501.05221v2}

\bibitem[Cla83]{MR710241}
Clarke, C. J.~S.
\newblock \emph{The cardinality of manifold atlases}.
\newblock Israel J. Math. \textbf{45} (1983)(1):9--16.
\newblock \doi{10.1007/BF02760666}.
\newblock \urlprefix\url{http://dx.doi.org/10.1007/BF02760666}

\bibitem[For11]{MR2975791}
Forstneri{\v{c}}, F.
\newblock \emph{Stein manifolds and holomorphic mappings}, \emph{Ergebnisse der
  Mathematik und ihrer Grenzgebiete. 3. Folge. A Series of Modern Surveys in
  Mathematics [Results in Mathematics and Related Areas. 3rd Series. A Series
  of Modern Surveys in Mathematics]}, vol.~56 (Springer, Heidelberg, 2011).
\newblock \doi{10.1007/978-3-642-22250-4}.
\newblock \urlprefix\url{http://dx.doi.org/10.1007/978-3-642-22250-4}

\bibitem[Gl{\"o}02]{MR1911979}
Gl{\"o}ckner, H.
\newblock \emph{Infinite-dimensional {L}ie groups without completeness
  restrictions}.
\newblock In \emph{Geometry and analysis on finite- and infinite-dimensional
  {L}ie groups ({B}\polhk edlewo, 2000)}, \emph{Banach Center Publ.}, vol.~55,
  pp. 43--59 (Polish Acad. Sci., Warsaw, 2002)

\bibitem[Gl{\"o}15]{HelgeSubmersionsAndImmersions}
Gl{\"o}ckner, H.
\newblock \emph{{Fundamentals of submersions and immersions between
  infinite-dimensional manifolds}} 2015.
\newblock \urlprefix\url{http://arxiv.org/abs/1502.05795}.
\newblock \eprint{1502.05795}

\bibitem[Hat02]{MR1867354}
Hatcher, A.
\newblock \emph{Algebraic topology} (Cambridge University Press, Cambridge,
  2002)

\bibitem[HM07]{MR2337107}
Hofmann, K.~H. and Morris, S.~A.
\newblock \emph{The {L}ie theory of connected pro-{L}ie groups}, \emph{EMS
  Tracts in Mathematics}, vol.~2 (European Mathematical Society (EMS),
  Z\"urich, 2007).
\newblock \doi{10.4171/032}.
\newblock \urlprefix\url{http://dx.doi.org/10.4171/032}.
\newblock A structure theory for pro-Lie algebras, pro-Lie groups, and
  connected locally compact groups

\bibitem[Jar81]{MR632257}
Jarchow, H.
\newblock \emph{Locally convex spaces} (B. G. Teubner, Stuttgart, 1981).
\newblock Mathematische Leitf{\"a}den. [Mathematical Textbooks]

\bibitem[Kel74]{keller1974}
Keller, H.
\newblock \emph{{Differential Calculus in Locally Convex Spaces}}.
\newblock Lecture Notes in Mathematics 417 (Springer Verlag, Berlin, 1974)

\bibitem[KK60]{MR0113228}
Kneser, H. and Kneser, M.
\newblock \emph{{Reell-analytische Strukturen der Alexandroff-Halbgeraden und
  der Alexandroff-Geraden.}}
\newblock Arch. Math. (Basel) \textbf{11} (1960):104--106

\bibitem[Kne58]{MR0101321}
Kneser, H.
\newblock \emph{Analytische {S}truktur und {A}bz\"ahlbarkeit}.
\newblock Ann. Acad Sci. Fenn. Ser. A. I. no. \textbf{251/5} (1958):8

\bibitem[Mil84]{MR830252}
Milnor, J.
\newblock \emph{Remarks on infinite-dimensional {L}ie groups}.
\newblock In \emph{Relativity, groups and topology, {II} ({L}es {H}ouches,
  1983)}, pp. 1007--1057 (North-Holland, Amsterdam, 1984)

\bibitem[Nee06]{MR2261066}
Neeb, K.-H.
\newblock \emph{Towards a {L}ie theory of locally convex groups}.
\newblock Jpn. J. Math. \textbf{1} (2006)(2):291--468.
\newblock \doi{10.1007/s11537-006-0606-y}.
\newblock \urlprefix\url{http://dx.doi.org/10.1007/s11537-006-0606-y}

\bibitem[Nyi92]{MR1164707}
Nyikos, P.~J.
\newblock \emph{Various smoothings of the long line and their tangent bundles}.
\newblock Adv. Math. \textbf{93} (1992)(2):129--213.
\newblock \doi{10.1016/0001-8708(92)90027-I}.
\newblock \urlprefix\url{http://dx.doi.org/10.1016/0001-8708(92)90027-I}

\bibitem[Rad25]{51.0273.01}
Rad\'o, T.
\newblock \emph{{\"Uber den Begriff der Riemannschen Fl\"ache.}}
\newblock Acta Szeged \textbf{2} (1925):101--121

\bibitem[Wal12]{MR2952176}
Walter, B.
\newblock \emph{Weighted diffeomorphism groups of {B}anach spaces and weighted
  mapping groups}.
\newblock Dissertationes Math. (Rozprawy Mat.) \textbf{484} (2012):128.
\newblock \doi{10.4064/dm484-0-1}.
\newblock \urlprefix\url{http://dx.doi.org/10.4064/dm484-0-1}

\bibitem[Whi36]{MR1503303}
Whitney, H.
\newblock \emph{Differentiable manifolds}.
\newblock Ann. of Math. (2) \textbf{37} (1936)(3):645--680.
\newblock \doi{10.2307/1968482}.
\newblock \urlprefix\url{http://dx.doi.org/10.2307/1968482}

\end{thebibliography}

\end{document}